\documentclass[11pt,a4paper,reqno]{amsart}
\usepackage{amsfonts}
\usepackage{amsmath,amssymb,amsthm,amsxtra}
\usepackage{float}
\usepackage[colorlinks, linkcolor=RoyalBlue,anchorcolor=Periwinkle,
    citecolor=Orange,urlcolor=Emerald]{hyperref}
\usepackage[usenames,dvipsnames]{xcolor}
\usepackage{enumitem}
\usepackage{geometry} 
\usepackage{graphicx}
\usepackage{subfigure}
\usepackage{bookmark}
\usepackage{tikz}\usetikzlibrary{matrix}
\usepackage{url}

\usepackage[all]{xy}

\theoremstyle{plain}
\newtheorem{theorem}{Theorem}[section]
\newtheorem{lemma}[theorem]{Lemma}
\newtheorem{corollary}[theorem]{Corollary}

\theoremstyle{definition}
\newtheorem{definition}[theorem]{Definition}
\newtheorem{example}[theorem]{Example}
\newtheorem{remark}[theorem]{Remark}
\numberwithin{equation}{section}

\def\hua{\mathcal}
\def\kong{\mathbb}
\def\<{\langle}
\def\>{\rangle}

\def\ZZ{\mathbb{Z}}

\def\Add{\operatorname{add}}
\def\Aut{\operatorname{Aut}}

\def\Sim{\operatorname{Sim}}
\def\Hom{\operatorname{Hom}}
\def\End{\operatorname{End}}
\def\Ext{\operatorname{Ext}}
\def\Irr{\operatorname{Irr}}
\def\tri{\triangle}

\def\diff{\operatorname{d}}
\def\Br{\operatorname{Br}}

\newcommand{\h}{\operatorname{\hua{H}}}            

\renewcommand{\k}{\mathbf{k}}
\renewcommand{\mod}{\operatorname{mod}}

\newcommand{\Cone}{\operatorname{Cone}}

\def\numbers{\begin{enumerate}[label=\arabic*{$^\circ$}.]}
\def\ends{\end{enumerate}}
\newcommand{\id}{\operatorname{id}}

\newcommand{\EGp}{\operatorname{EG}^\circ}       
\newcommand{\C}[2]
{\operatorname{\hua{C}}_{#1}(#2)}               

\newcommand{\CEG}[2]{\operatorname{CEG}_{#1}(#2)}             
\newcommand{\D}{\operatorname{\hua{D}}}

\newcommand{\per}{\operatorname{per}}

\def\zero{\hua{H}_\Gamma}

\def\arrow{red}
\usepackage{tikz,ifthen}
\def\surf{\mathbf{S}}                       
\def\Dfd{\D(\PP)}

\newcommand{\MCG}{\operatorname{MCG}}

\newcommand{\EGT}{\operatorname{EG}^{\bowtie}}
\newcommand{\PP}{\Gamma_{\surf}}             
\newcommand\Dehn[1]{\mathrm{D}_{#1}}

\renewcommand\bowtie{\times}
\def\st{\mathbf{T}_{\bowtie}}

\def\dexc{red}

\def\T{\mathbf{T}}
\def\stab{\operatorname{stab}}

\usetikzlibrary{arrows}

\def\TA{{\mathbf{A}^\bowtie}(\surf)}
\def\obj{\hua{C}^\bowtie(\surf)}
\def\J{{J}}
\def\QP{\Gamma(Q,W)}
\def\M{\mathbf{M}}
\def\P{\mathbf{P}}
\def\PM{\MCG_{\bullet}}
\def\TM{\MCG_{\bowtie}}

\usepackage{yhmath}

\title{Tagged mapping class groups: Auslander-Reiten translation}
\author{{\small Thomas Br\"{u}stle  \and Yu Qiu }}
\thanks{Supported by the NSERC of Canada and Bishop's University}
\date{\today}
\begin{document}
\begin{abstract}
    We give a geometric realization, the tagged rotation,
    of the AR-translation on the generalized cluster category associated to a surface $\surf$ with marked points and non-empty boundary,
    which generalizes Br\"{u}stle-Zhang's result for the puncture free case.

    As an application, we show that the intersection of
    the shifts in the 3-Calabi-Yau derived category $\Dfd$ associated to the surface
    and the corresponding Seidel-Thomas braid group of $\Dfd$ is empty,
    unless $\surf$ is a polygon with at most one puncture (i.e. of type A or D).

    \vskip .3cm
    {\parindent =0pt
    \it Key words:} Mapping class group, Auslander-Reiten translation,
    triangulated surface, cluster theory, braid group
\end{abstract}

\maketitle

\section{Introduction}

Fomin, Shapiro and Thurston studied in  \cite{FST} the cluster combinatorics of a marked surface,
that is, an oriented surface $\surf$ with a finite set of marked points $\M$ on the boundary and a finite set of punctures $\P$ inside $\surf$. We do not consider closed surfaces in this paper, that is we assume $\surf$ to have a non-empty boundary.
Studying the relation between Teichm\"uller theory and  cluster algebras, it has been shown in \cite{GSV,FG1,FG2} that the flip of an arc in a (ideal) triangulation corresponds to a mutation of a cluster variable. However, self-folded triangles do not permit to flip the internal arc, whereas the mutation of a cluster variable is always defined.
In order to allow for flips at all arcs,  Fomin, Shapiro and Thurston introduce in  \cite{FST} the concept of tagged triangulations, which are (ideal) triangulations equipped with a tagging at each puncture. The thus obtained exchange graph $\EGT(\surf)$ of tagged triangulations, encoding flips of tagged triangulations, is isomorphic to the cluster exchange graph of the cluster algebra ${\mathcal A}(\surf)$ defined by the marked surface  $\surf$.

The aim of this paper is to compare geometric properties of the marked surface $\surf$ with properties of the categorification $\C{}{\surf}$ of the cluster algebra ${\mathcal A}(\surf)$:
Labardini-Fragoso associated in \cite{LF} a quiver with potential to the marked surface $\surf$ which allows to use
Amiot's  construction in \cite{A}  to define a 2-Calabi-Yau category $\C{}{\surf}$.
It is  shown in \cite{CKLP}  that the cluster-tilting objects in $\C{}{\surf}$ correspond bijectively to
the clusters of the cluster algebra ${\mathcal A}(\surf)$.
Moreover,
the set $\TA$ of tagged arcs in $\surf$ corresponds bijectively to the set $\obj$ of reachable indecomposable objects
in $\C{}{\surf}$, and we denote by $$\Aut_0\C{}{\surf}=\Aut\obj/\stab\obj$$
the group of auto-equivalences of $\C{}{\surf}$ that preserve the set of reachable indecomposable objects
modulo the subgroup of autoequivalences that act as identity on $\obj$.
We view this group as an analogue of the group of cluster algebra automorphisms studied in \cite{ASS}.

On the geometric side, we consider the group of orientation-preserving  diffeomorphisms of $\surf$ that preserve (not necessarily pointwise)
the set of marked points $\M$, the set of punctures $\P$ and the boundary.
The marked mapping class group $\PM(\surf)$ of $\surf$ is the quotient of this group by the subgroup of diffeomorphisms that are isotopic to the identity.
Denote by
$$\TM(\surf) = \PM(\surf)\ltimes(\kong{Z}_2)^p$$
 the tagged mapping class group of $\surf$
consisting of elements $(\varphi,\delta)$, where
$\varphi\in\PM(\surf)$ and $\delta$ is a $\{\pm1\}$-sign on each point in $\mathbf{P}$.
An element in $\TM(\surf)$ acts on the set $\TA$ and hence on  $\obj$, which induces an embedding
\begin{gather}
 \TM(\surf) \hookrightarrow \Aut_0\C{}{\surf} \label{embed}
 \end{gather}
The category  $\C{}{\surf}$ admits a distinguished automorphism, the Auslander Reiten translation, which in this context coincides with the suspension functor of the triangulated category $\C{}{\surf}$, see \cite{KR}.
The first aim of this article is to show that the Auslander-Reiten translation can be realized by an element in the tagged mapping class group  $\TM(\surf)$.
For each boundary component $Y$ with $m$ marked points, denote by $\rho_Y$ the $m$-th root of the Dehn twist around $Y$, that is, simultaneous rotation to the next marked point on $Y$. Then we define
the (universal) tagged rotation $\varrho$ as the permutation on $\TA$
induced by the element
\begin{gather*}
    \varrho=\prod_{Y\subset\partial\surf} \rho_Y \cdot \prod_{P\in\P} \delta^P
\end{gather*}
in $\TM(\surf)$ where the first product is over all connected components $Y$ of $\partial\surf$ and the second product
is a simultaneous change of tagging.
As one of the main results of this paper we show in Theorem \ref{thm:T-rotation} that the tagged rotation $\varrho\in\TM(\surf)$ on $\TA$ becomes the shift $[1]$ on $\obj$.

This result is known from \cite{Sch} for the case of a punctured disc and from \cite{BZ} for all unpunctured surfaces. In fact, the cluster category of an unpunctured surface is explicitly given by a combinatorial description of string and band objects, and the shift functor $[1]$ is well-known from the theory of string algebras.
These methods are not available in the general situation
since there is no explicit description of the (indecomposable) objects and morphisms yet.
Instead, we show that one can always choose a triangulation such that the operation of the shift functor in the category  $\C{}{\surf}$ corresponds to the tagged flip of an arc.
The proof is thus an interplay between the triangulated structure of the category $\C{}{\surf}$ and the structure of the triangulations of the surface $\surf$.
In fact, the result we obtain helps to understand the indecomposable objects and irreducible  morphisms of $\C{}{\surf}$, since these are encoded in the Auslander-Reiten triangles of the triangulated category $\C{}{\surf}$.
A further study aiming to generalize results in \cite{BZ} to
the unpunctured case is undertaken in \cite{QZ}.

One motivation to study such a geometric realization of the shift functor comes from physics: to compute the complete spectrum of a BPS particle, one studies maximal green sequences which go from a triangulation to its shift in $\C{}{\surf}$, see \cite{CCV:Braids}. Our result allows to determine the tagged triangulation for the endpoint
of any maximal green mutation sequence from the tagged triangulation of the starting point,
see \cite{BDP} (cf. \cite{Q3}). Note that we do not prove the existence of such a (finite) maximal green sequence, we can only provide a guidance in the search for maximal green sequences: if one exists, it needs to end in the triangulation obtained by applying the tagged rotation $\varrho$.

An invariant induced by the tagged rotation is the \emph{order} of a tagged arc $\alpha$,
that is, the minimal natural number $m$ satisfying $X_\alpha[m]=X_\alpha$,
where $X_\alpha$ is the object corresponding to $\alpha$ in $\C{}{\surf}$.
By convention, the order is infinite if such an $m$ does not exist.
We study in the last section of this article the order $m$ of the shift operator $[1]$ of $\C{}{\surf}$, and compare the action of the shift functor to the action of the Seidel-Thomas braid group $\Br\surf$ which is
generated by spherical twists.
We show in Theorem~\ref{thm:shift} that the intersection of the shifts and $\Br\surf$ is empty,
unless $\surf$ is a polygon with at most one puncture (i.e. of type A or D).

Having realized one distinguished automorphism of $\C{}{\surf}$ geometrically,
we further show  (Proposition~\ref{thm:auto}) that every cluster automorphism
is induced from elements in the corresponding tagged mapping class group
except for three cases (the marked surfaces corresponding to type $D_2$, $D_4$ and $\widetilde{D}_4$).

\subsection*{Acknowledgements}
Both authors thank the referee for numerous comments that helped improving a first version of the paper. Moreover, Y. Qiu would like to thank B. Keller for sharing his expertise on cluster theory.

\section{Preliminaries}\label{Preliminaries}
\subsection{Quiver with potential from surface}
Throughout the article, $\surf$ denotes a \emph{marked surface} in the sense of \cite{FST},
that is, a connected Riemann surface with a fixed orientation endowed
with a finite set of marked points $\M$ on the boundary $\partial\surf$
and a finite set of punctures $\P$ inside $\surf$
such that each connected component of the boundary of $\surf$
contains at least one marked point.
Unless otherwise stated, we will always suppose that $\partial\surf \neq \emptyset.$

We recall the following definitions from \cite{FST}, see there for the precise conditions for everything being well-defined:
Curves in $\surf$ are considered up to isotopy with respect to the sets of marked points and punctures.
Two curves are called \emph{compatible} if they do not intersect in their relative interior.
A \emph{simple} curve has no self-intersection except possibly on the endpoints, and
an \emph{arc} in $\surf$ is a simple curve whose both endpoints are marked points or punctures
and which is not isotopic to a boundary component.

An \emph{ideal triangulation} $\T$ of $\surf$ is
a maximal collection of pairwise compatible arcs in $\surf$.
Any triangulation $\T$ of $\surf$ consists of
\begin{gather}\label{eq:n}
n=6g+3p+3b+m-6
\end{gather}
arcs (ordinary or tagged), where $g$ is the genus of $\surf$, $b$ is the number of boundary components and $m$ and $p$ denote the number of marked points and punctures.
The number $n$ is called the \emph{rank} of the surface $\surf$.
To exclude a few cases where the surface does not admit a triangulation we assume from now on
that $n \ge 1.$

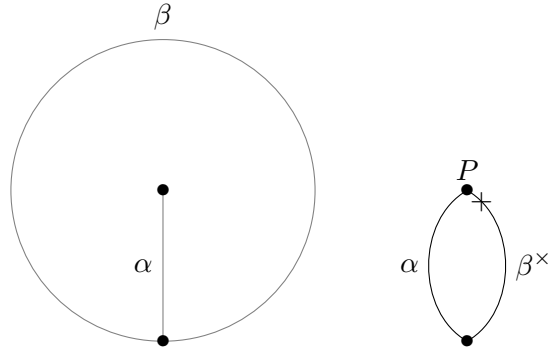
\begin{figure}[ht]\centering
\begin{tikzpicture}
\draw[gray,thin] (0,0) circle (2) to (0,-2);
\draw (0,0) node {$\bullet$} (0,-2) node {$\bullet$};
\draw (0,2) node[above] {$\beta$};\draw (0,-1) node[left] {$\alpha$};
\draw(4,-2)to[bend left=60](4,0);\draw (3.5,-1) node[left] {$\alpha$};
\draw(4,-2)to[bend right=60](4,0);
\draw (4.5,-1) node[right] {$\beta^{\bowtie}$};
\draw(4,-2)node{$\bullet$}
(4,0)node{$\bullet$}(4.19,-.15)node{$+$}(4,0)node[above]{$P$};
\end{tikzpicture}
\caption{A self-folded triangle and the corresponding tagged version}
\label{fig:self-folded}
\end{figure}

The \emph{flip of an arc} $\alpha$ in a triangulation $\T$ is the unique arc $\alpha' \neq \alpha$ which forms a triangulation with the remaining arcs of $\T$.
We denote by $\EGp(\surf)$  the \emph{exchange graph of ideal triangulations}
which is formed by the ideal triangulations of $\surf$, with an edge between $\T$ and $\T'$ whenever $\T'$ is obtained from $\T$ by the flip of an arc.
An arc in an ideal triangulation can always be flipped except when it is the arc $\alpha$ connecting to the internal puncture of a self-folded triangle as in the left side of figure \ref{fig:self-folded}.
To overcome this shortcoming, Fomin, Shapiro and Thurston introduced in \cite{FST} the concept of tagging which allows to distinguish two ways an arc can end in a puncture (tagged or untagged). They subsequently develop the concepts of tagged triangulations, tagged flips and the exchange graph of tagged triangulations, which is denoted by $\EGT(\surf)$.
Fomin, Shapiro and Thurston note that $\EGT(\surf)$ can be obtained by gluing $2^p$ copies of $\EGp(\surf)$.

\begin{definition}\label{def-potential}
We have the following data associated to
a tagged triangulation $\st$:
\begin{itemize}
\item The signed adjacency matrix $\mathrm{B}=\mathrm{B}(\st)$.
The rows and columns of $B$ are naturally labeled by the arcs in $\st$
(say from $1$ to $n$).
It is skew-symmetric and all its entries $b_{ij}$ are in $\{0,\pm1,\pm2\}$, see \cite[Definition~4.1]{FST} for more details.
\item The quiver $Q=Q(\st)$ corresponding to $\mathrm{B}(\st)$  is the quiver whose vertices are the arcs in $\st$ and
the number of arrows from $i$ to $j$ equals $b_{ij}$.
For instance, 
the quiver for a triangle is shown in Figure~\ref{fig:first}.
\item The potential $W=W(\st)$  is defined as a sum
of certain cycles in the complete path algebra $\widehat{\k Q}$,
where $\k$ is an algebraic closed field, see \cite{CLF}.
If $\partial\surf\neq\emptyset$, this potential is rigid (and thus non-degenerate) by
\cite{LF}.
\end{itemize}
\end{definition}

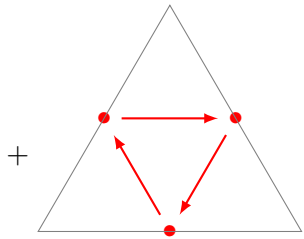
\begin{figure}[t]\centering
  \begin{tikzpicture}
  \draw(-2,0)node{+};
  \foreach \j in {1,...,3}  { \draw (120*\j-30:2) coordinate (v\j);}
    \path (v1)--(v2) coordinate[pos=0.5] (x3)
              --(v3) coordinate[pos=0.5] (x1)
              --(v1) coordinate[pos=0.5] (x2);
    \foreach \j in {1,...,3}{\draw (x\j) node[red] (x\j){$\bullet$};}
    \draw[->,>=latex,red,thick] (x1) to (x3);
    \draw[->,>=latex,red,thick] (x3) to (x2);
    \draw[->,>=latex,red,thick] (x2) to (x1);
    \draw[gray,thin] (v1)--(v2)--(v3)--cycle;
  \end{tikzpicture}
\caption{The quiver associated to a triangle}
\label{fig:first}
\end{figure}

\subsection{Jacobian algebra and Ginzburg algebra}
Let $Q$ be a finite quiver and $W$ a potential on $Q$. The \emph{Jacobian algebra} of the quiver with potential $(Q, W)$,
denoted by $\J(Q, W)$, is the quotient of the complete path algebra
$\widehat{\k Q}$ by the closure of the ideal generated by $\partial_a W$,
where $a$ runs over all arrows of $Q$.
The pair $(Q, W)$ is called \emph{Jacobi-finite} provided $\J(Q, W)$ is finite-dimensional.
The quiver with potential associated to a triangulation of a surface with non-empty boundary is
Jacobi-finite by \cite{CLF}.
The Jacobian algebra can be seen as the zeroth cohomology of the \emph{Ginzburg dg algebra} $\Gamma=\QP$  of $(Q, W)$
which is constructed as follows (\cite[Section~7.2]{K10}):

Let $\overline{Q}$ be the graded quiver with the same vertices as $Q$
and whose arrows are the arrows in $Q$ in degree $0$,
an arrow $a^*:j\to i$ with degree $-1$ for each arrow $a:i\to j$ in $Q$ and
a loop $e^*:i\to i$ with degree $-2$ for each vertex $e$ in $Q$.

The underlying graded algebra of $\QP$ is the completion of
the graded path algebra $\k \overline{Q}$.
The differential $d$ of $\QP$ is the unique continuous linear endomorphism homogeneous
of degree $1$ which satisfies the Leibniz rule and
takes the following values on the arrows of $\overline{Q}$:
\begin{gather*}
    \begin{cases}
    \diff a=0,\quad
    \diff a^*=\partial_a W,&\forall a\in Q_1,\\
    \diff \sum_{e\in Q_0} e^*=\sum_{a\in Q_1} \, [a,a^*].
    \end{cases}
\end{gather*}
Denote by $\D_{fd}(\Gamma)$ and $\per(\Gamma)$
the \emph{finite dimensional derived category} of $\QP$ and
the \emph{perfect derived category }of $\QP$, respectively (cf. \cite[Section~7.3]{K10}).
In case $(Q,W)$ is Jacobi-finite, Amiot introduced in \cite{A} the \emph{generalized cluster category} $\C{}{\Gamma}$ as the quotient category $\per(\Gamma)/\D_{fd}(\Gamma)$.
In other words, there is an exact sequence of triangulated categories
\begin{gather}\label{eq:ses}
    0 \to \D_{fd}(\Gamma) \to  \per(\Gamma) \to \C{}{\Gamma} \to 0.
\end{gather}

A \emph{cluster tilting set} $\hua{P}=\{P_j\}_{j=1}^n$ in $\C{}{\Gamma}$ is
an Ext-configuration, i.e. a maximal collection
of non-isomorphic indecomposables such that $\Ext_{\C{}{\Gamma}}^1(P_i, P_j)=0$.
Note that $n$ is the number of vertices in $Q$.
The \emph{mutation} $\mu_i$ at the $i$-th object
acts on a cluster tilting set $\{P_j\}_{j=1}^n$ by replacing $P_i$ by
another object $P_i'$ which can be calculated as follows: Define
     \begin{gather}
         \label{eq:mutate1}
        P_i^\sharp = \Cone(P_i \to \bigoplus_{j\neq i} \Irr(P_i,P_j)^*\otimes P_j)
           \end{gather}
    and
    \begin{gather}\label{eq:mutate2}
      P_i^\flat =  \Cone(\bigoplus_{j\neq i} \Irr(P_j,P_i)\otimes P_j \to P_i)[-1]
    \end{gather}
    where $\Irr(P_i,P_j)$ is the space of irreducible maps $P_i\to P_j$
    in the additive subcategory $\Add \mathbf{P}$ of $\C{}{\Gamma}$,
    for $\mathbf{P}=\bigoplus_{i=1}^n P_i$.
    One then checks that the two objects $ P_i^\sharp$ and $P_i^\flat$ are isomorphic and defines $P_i'$ to be either of them, see \cite{A}.

    Obviously, the indecomposable summands of $\Gamma$ form a cluster tilting set of $\C{}{\Gamma}$. An arbitrary cluster tilting set of $\C{}{\Gamma}$ is called  \emph{reachable} if it can be obtained from $\Gamma$ by a finite sequence of mutations. An \emph{indecomposable reachable object} of $\C{}{\Gamma}$ is an object that is contained in a reachable cluster tilting set.

A key property of mutation is that it is an involution, i.e. we have
\[\mu_i^2(\hua{P})=\hua{P}\quad \mbox{ for all possible } i.\]

If $\Gamma=\Gamma(Q,W)$ is given by the quiver and potential associated to a triangulation  $\T$ of $\surf$, then the category  $\D_{fd}(\Gamma)$ and the cluster category $\C{}{\Gamma}$ are (up to derived equivalence) independent of the choice of the triangulation $\T$: as shown in \cite{FST}, any two tagged triangulations of a surface with non-empty boundary are connected by a sequence of flips. In \cite{LF2} it is further shown that the corresponding quivers with potential are also related by a sequence of mutations, and it follows then from \cite{K10} that the corresponding categories $\D_{fd}(\Gamma)$   are equivalent as triangulated categories, likewise for $\C{}{\Gamma}$.
We therefore write simply $\D(\surf)= \D_{fd}(\Gamma)$ and $\C{}{\surf}=\C{}{\Gamma}$ in this case.
Moreover, we denote by $\CEG{*}{\surf}$ the
 \emph{exchange graph}  of reachable cluster tilting sets of $\C{}{\surf}$, that is,
the connected unoriented graph  whose vertices are the reachable cluster tilting sets
and whose edges are the mutations. We refer to \cite{BY} for more details on and equivalent definitions of the graph $\CEG{*}{\surf}$.
Let $\obj$ be the set consisting of objects that appear
in some cluster tilting set $\hua{P}$ in $\CEG{*}{\surf}$.

\subsection{The canonical bijection}
By the correspondence between tagged arcs and cluster variables
(\cite[Theorem~7.11]{FST}) and
the correspondence between cluster variables and indecomposable reachable objects
(\cite[Corollary~3.5]{CKLP}),
we get a  bijection between $\TA$ and $\obj$. Both correspondences are defined using a fixed triangulation, but are shown to commute with mutations. We therefore obtain the following lemma.

\begin{lemma}\label{lem:bij}
There is a canonical bijection
\begin{gather}\label{eq:zeta}
    \zeta\colon\TA\to\obj
\end{gather}
which induces the isomorphism (between graphs)
\begin{gather}\label{eq:EGT}
  \EGT(\surf)\cong\CEG{*}{\surf},
\end{gather}

which sends a tagged triangulation $\st$ to
the cluster tilting set consisting of objects $\{ \zeta(\alpha) \mid \alpha\in\st\}$
and a tagged flip in $\EGT(\surf)$ to a mutation in $\CEG{*}{\surf}$.
\end{lemma}

Note that we are using in this lemma the fact that the surface has non-empty boundary, and/or the definition of the exchange graph given by reachable cluster tilting sets: The mutation graph of all cluster tilting sets has two connected components for the case of a once punctured torus without boundary, and thus is not isomorphic to the (connected) graph $\EGT(\surf)$.

\section{The tagged rotation}
\subsection{Tagged mapping class group}
Consider the group of orientation-preserving  diffeomorphisms of $\surf$ that preserve (not necessarily pointwise)
the set of marked points $\M$, the set of punctures $\P$ and the boundary.
The quotient of this group by the subgroup of diffeomorphisms that are isotopic to the identity
through isotopies that fix the set of marked points pointwise is called the \emph{marked mapping class group} $\PM(\surf)$ of $\surf$.
As in \cite{ASS}, we denote by
\[\TM(\surf) = \PM(\surf)\ltimes(\kong{Z}_2)^p\]
the \emph{tagged mapping class group} of $\surf$
consisting of elements $(\varphi,\delta)$ where
$\varphi\in\PM(\surf)$ and $\delta$ is a $\{\pm1\}$-sign on each point in $\mathbf{P}$,
endowed with  the multiplication defined by
\[
    (\varphi', \delta') \circ (\varphi, \delta)
    =\Big(\varphi'\circ\varphi, (\delta'\circ\varphi)\cdot\delta\Big)
\]
using the action \[ \left((\delta'\circ\varphi)\cdot\delta\right)(P)=\delta'(\varphi(P))\cdot\delta(P).\]
We can identify $\PM(\surf)$ with
\[
    \PM(\surf)\times \mathbf{1}\subset\TM(\surf),
\]
where $\mathbf{1}(P)=1$ for any $P\in\P$.

Denote by $\Aut\C{}{\surf}$ the group of auto-equivalences of $\C{}{\surf}$
and by $\stab\obj$ the subgroup of $\Aut\C{}{\surf}$ that preserves (pointwise) the set $\obj$.
Let $\Aut_0\C{}{\surf}=\Aut\C{}{\surf}/\stab\obj$.

Any element in $\TM(\surf)$ acts on the set $\TA$ and hence on  $\obj$ via
the canonical bijection $\zeta$ in Lemma~\ref{lem:bij}.
The following lemma is an analogue of \cite[Theorem 4.11]{ASS} in terms of the cluster category.
Note however that one cannot directly use the result from \cite{ASS} to define a group homomorphism from $ \TM(\surf)$ to $\Aut_0\hua{C}(\surf)$ since the cluster variables correspond only to (reachable) rigid objects, whereas the cluster category can contain more (in fact infinite families of)  indecomposable objects.

\begin{lemma}\label{lem:MCG}
There is a canonical injection
$\iota_\surf\colon\TM(\surf)\hookrightarrow\Aut_0\C{}{\surf}$.
\end{lemma}

\begin{proof}
We first recall  how the cluster category $\C{}{\surf}$ is defined.
Given a triangulation $\T$, there is an associated quiver with potential $(Q_\T, W_\T)$
with  corresponding Ginzburg dg algebra $\Gamma_\T$.
Then $\C{}{\T}$ is defined to be the quotient category $\per\Gamma_\T/\D_{fd}(\Gamma_\T)$
and there is a corresponding bijection $\zeta_\T$ in \eqref{eq:zeta}.
For any two triangulations $\T$ and $\T'$ related by a flip $f$, the corresponding quivers with potential are right equivalent by \cite[Theorem 8.1]{LF2}, and
there are two canonical derived equivalences (see, e.g. \cite[Theorem 7.4]{K10})
\[
    \Phi_{\pm}\colon \per\Gamma_\T\cong\per\Gamma_{\T'}.
\]
Moreover, the  composition $\Phi_+^{-1}\circ\Phi_-$ yields a spherical twist
(cf. Definition~\ref{def:twist})
and, by \cite[Theorem 5.6]{K6}, the two functors induce the same
triangulated equivalence
$$E_f\colon\C{}{\T}\cong\C{}{\T'}.$$
Further, the bijections $\zeta_\T$ and $\zeta_{\T'}$ are compatible under this equivalence, i.e.
\begin{gather}\label{eq:zetastable}
 E_f\left( \zeta_\T(\gamma) \right)=\zeta_{\T'}(\gamma).
\end{gather}

Given two triangulations $\T$ and $\T'$ and a sequence $s$ of flips that relates them,
there is an equivalence $E_s$ between $\C{}{\T}$ and $\C{}{\T'}$ defined as the composition of the equivalences $E_f$ where $f$ runs through the flips in the sequence $s$.
For any two such sequences $s$ and $s'$ transforming $\T$ into $\T'$, we have $E^{-1}_{s'}\circ E_s = \rm{id}  \in\Aut_0\hua{C}(\T)$ since they act identically on $\hua{C}^\bowtie(\T)$ by \eqref{eq:zetastable}.
Thus we can identify all the cluster categories $\C{}{\T}$ as $\C{}{\surf}$, and this identification $E_\T : \C{}{\T} \to \C{}{\surf}$ is unique when studying elements in the quotient group   $\Aut_0\C{}{\surf}$.

Now consider an element $\varphi\in\TM(\surf)$ which takes a triangulation $\T$ to $\T'$.
So we have the following canonical identifications as above:
\begin{gather}\label{eq:EE}\xymatrix{
    \C{}{\T}\ar@{--}[rr]^{E_\varphi}\ar[dr]_{E_\T}&&\C{}{\T'}\ar[dl]^{E_{\T'}}\\&\C{}{\surf}.
}\end{gather}
As $\T'=\varphi(\T)$, the associated quivers with potential are right equivalent
(in particular, the associated quivers are isomorphic in a canonical way),
which induces an isomorphism $\Gamma_\T\to\Gamma_\T'$, see \cite[Proposition~3.5]{KY} for such type of isomorphisms.
Consider the corresponding derived equivalence $i_\varphi\colon\per\Gamma_\T\to\per\Gamma_{\T'}$
which induces an equivalence $E_\varphi\colon\C{}{\T}\cong\C{}{\T'}$.
Then $E_\varphi$ maps the canonical cluster tilting set $\mathbf{M}_\T$ in $\C{}{\T}$ to
the canonical cluster tilting set $\mathbf{M}_{\T'}$ in $\C{}{\T'}$.
Note that  $\mathbf{M}_{\T}$ and  $\mathbf{M}_{\T'}$ are in fact the images of $\Gamma_\T$ and $\Gamma_{\T'}$ under the projection
from the perfect derived category to the cluster category.
In general these two cluster tilting sets do not map under the canonical identifications $E_\T$ and $E_{\T'}$ to the same cluster tilting set in $\C{}{\surf}$
in \eqref{eq:EE}.

Moreover, since $E_\varphi$ preserves the canonical cluster tilting sets
 it also preserves the set of reachable rigid objects.
Define $\iota_\surf(\varphi)$ to be the auto-equivalence
$$ E_{\T'}  \circ E_\varphi\circ  E_\T^{-1}  ,$$
which is well-defined as an element in $\Aut_0\C{}{\surf}$.

The injectivity of $\iota_\surf $ follows from the well-known Alexander's method from topology.
More precisely, the element in $\TM(\surf)$ that preserves a (tagged) triangulation
must be the identity.
\end{proof}

The fixed orientation $\hua{O}$ of $\surf$ induces
an orientation $\hua{O}$ on each component of $\partial\surf$ such that the surface $\surf$ lies to the left of every component.
This allows to define a rotation $\rho$  on the set of marked points $\mathbf{M}$ in $\partial\surf$
which sends each marked point $M$ to
the next marked point  $M'$ in $\partial\surf$ along $\hua{O}$ (cf. Figure~\ref{fig:rotate}).
We would like to extend this rotation to an element in $\PM(\surf)$:

\begin{figure}[b]\centering
\begin{tikzpicture}[scale=.4]
\draw[thick](0,0) circle (5) (-135:4)node{$+$};
\draw[thick,fill=gray!20](0,0) circle (1.5) (126:5.9)node{$\rho$}(126:.5)node{$\rho$};
  \foreach \j in {1,...,5}{\draw(90-72*\j:5) node{$\bullet$};}
  \foreach \j in {1,...,3}{\draw(-90-120*\j:1.5) node{$\bullet$};}

\draw[blue,->,>=latex](90+36-15:5.4)arc(90+36-15:90+36+15:5.4);
\draw[blue,->,>=latex](90+36+25:1.1)arc(90+36+25:90+36-45:1.1);
\end{tikzpicture}
\caption{The rotation on $\surf$}
\label{fig:rotate}
\end{figure}
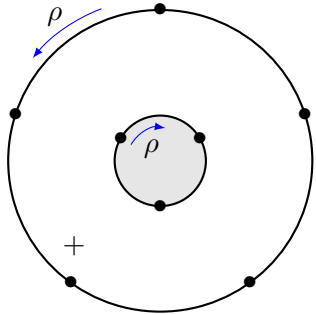

\begin{example}\label{ex:rotation}
For each component $Y$ in the boundary $\partial\surf$ with $m_Y=|\M \cap Y|$ marked points
there is a rotation $\rho_Y$ in $\PM(\surf)\subset\TM(\surf)$ which sends each marked point $M$ in $Y$ to
the next marked point  in $Y$.
The $m_Y-$th power $\rho_Y^{m_Y}$ of the rotation $\rho_Y$ coincides with
the (positive) Dehn twist $\Dehn{Y}$ along $Y$ (cf. Figure~\ref{fig:The Dehn twist} and see \cite{FM} for a definition of Dehn twists).
As an element in the mapping class group $\PM(\surf)$, $\rho_Y$ acts on the set of arcs of $\surf$ and thus induces a permutation on $\TA$
where the tagging of $\rho_Y(\alpha)$ at any puncture equals
the tagging of the arc $\alpha$ at that puncture.
\end{example}

\begin{figure}[t]\centering
\begin{tikzpicture}[xscale=.3,yscale=.3]
  \draw[thick](0,0)circle(7);
  \draw[thick, fill=gray!45](0,0)circle(1);
  \draw[\arrow, very thick](0,0)circle(4);
  \draw(60:4.8)node{$C$}  (-120:5.5)node{+};
  \draw[smooth,domain=1:7,blue,very thick,dashed] plot (90:\x);
  \draw[smooth,domain=1:7,cyan,very thick,dashed] plot (180+90:\x);
  \draw(0:7.5)edge[very thick,->,>=latex](0:11);\draw(0:9)node[above]{$\Dehn{C}$};
\end{tikzpicture}
\begin{tikzpicture}[xscale=.3,yscale=.3]
  \draw[thick](0,0)circle(7);
  \draw[thick, fill=gray!45](0,0)circle(1);
  \draw[\arrow, very thick](0,0)circle(4);
  \draw[smooth,domain=1:7,blue,very thick,dashed] plot (150-\x*60:\x);
  \draw[smooth,domain=1:7,cyan,very thick,dashed] plot (-30-\x*60:\x);
\end{tikzpicture}
\caption{The Dehn twist}
\label{fig:The Dehn twist}
\end{figure}
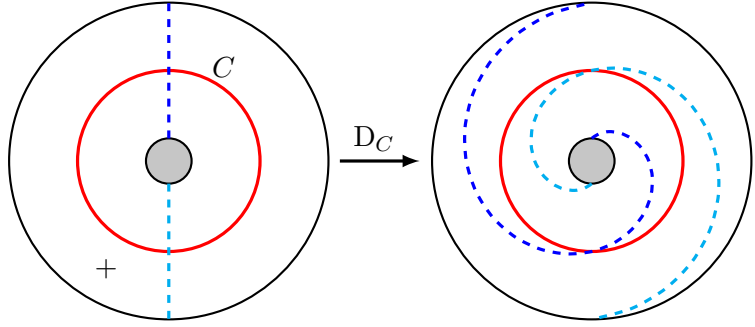

\begin{example}\label{ex:rotation2}
For each puncture $P\in\P$, there is a \emph{tagging switch} $\delta^P=(\id, \delta_P)$
in $\TM(\surf)$, such that $\delta_P(P')=-1$ if and only if $P'=P$.

Further, $\delta^P$ induces a permutation on the set of tagged arcs $\TA$
that preserves the underlying ordinary arcs but changes the tagging
at the puncture $P$.
\end{example}

\begin{definition}\label{def:T-rotate}
The \emph{(universal) tagged rotation} $\varrho$ is the permutation on $\TA$
induced by the element
\begin{gather}\label{eq:T-rotation}
    \varrho=\prod_{Y\subset\partial\surf} \rho_Y \cdot \prod_{P\in\P} \delta^P
\end{gather}
in $\TM(\surf)$ where the first product is over all connected components $Y$ of $\partial\surf$. Note that the order in \eqref{eq:T-rotation} does not matter since the $\rho_Y$ commute.
\end{definition}

An instance of a tagged rotation is shown in Figure~\ref{fig:T-rotate}.

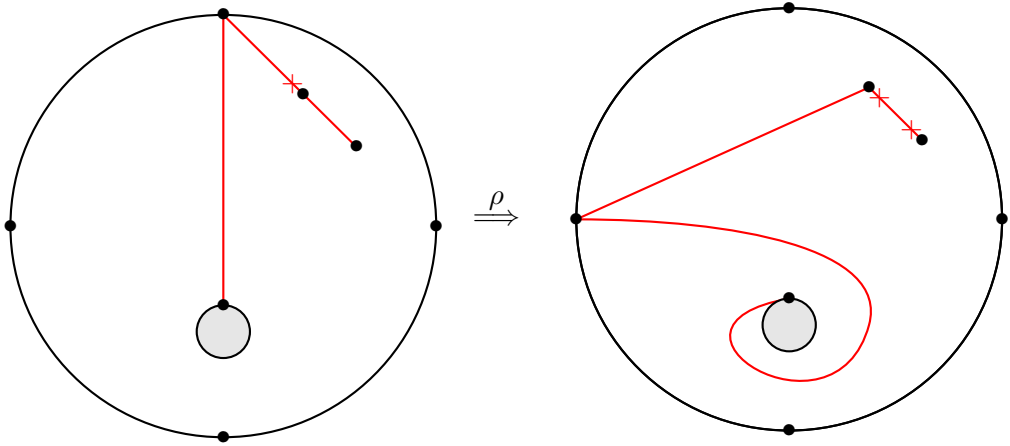
\begin{figure}[t]\centering
\begin{tikzpicture}[scale=.7]
\draw[\dexc,thick](0,4)to(1.5,2.5)to(2.5,1.5);
\draw[\dexc,thick] (0,4)to(0,-1.5);
\draw[thick](0,0) circle (4)(1.5,2.5)node{$\bullet$}(2.5,1.5) node{$\bullet$};
\draw[thick,fill=gray!20](0,-2) circle (.5) (0,-1.5) node{$\bullet$}
    (1.3,2.7) node[\dexc]{$+$};
  \foreach \j in {1,...,4}{\draw(90-90*\j:4) node{$\bullet$};}
\end{tikzpicture}
\begin{tikzpicture}[scale=.7]
\draw[\dexc,thick](-4,0)to(1.5,2.5)to(2.5,1.5);
\draw[\dexc,thick] (-4,0).. controls +(0:3) and +(75:1.5) ..(1.5,-2)
    .. controls +(-105:2.5) and +(190:3) ..(0,-1.5);
\draw[thick](1.7,2.3) node[\dexc]{$+$}  (2.3,1.7) node[\dexc]{$+$};;
\draw(-5.5,0)node{$\Longrightarrow$}node[above]{$\rho$};
\draw[thick](0,0) circle (4);
\draw[thick](0,0) circle (4)(1.5,2.5)node{$\bullet$}(2.5,1.5) node{$\bullet$};
\draw[thick,fill=gray!20](0,-2) circle (.5)(0,-1.5) node{$\bullet$};
  \foreach \j in {1,...,4}{\draw(90-90*\j:4) node{$\bullet$};}
\end{tikzpicture}
\caption{The tagged rotation on arcs in $\surf$}
\label{fig:T-rotate}
\end{figure}

\subsection{AR-translation on marked surfaces}
This subsection is devoted to show that the tagged rotation corresponds to the shift
in the cluster category via the bijection $\zeta$ in Lemma~\ref{lem:bij}.

\begin{lemma}\label{lem:T-rotation}
For a tagged arc $\alpha\in\TA$
whose end points are distinct and at least one of which is not a puncture,
we have $\zeta(\varrho(\alpha))=\zeta(\alpha)[1]$.
\end{lemma}
\begin{proof}
By formula \eqref{eq:mutate1}, for any cluster tilting set $\hua{P}=\{P_j\}_{j=1}^n$,
an object $P_i$ in $\hua{P}$ becomes $P_i[1]$ in $\mu_i(\hua{P})$ if
\begin{gather}\label{eq:irr}
  \Irr(P_i,P_j)=0, \quad \forall j.
\end{gather}
Since the dimensions of the spaces of irreducible morphisms between the objects $P_k$ correspond to the numbers of arrows between vertices in the Gabriel quiver $Q_\hua{P}$ of $\End(\oplus_{k=1}^n P_k)^{\rm op}$
condition \eqref{eq:irr} is equivalent to $P_i$ being a source in $Q_\hua{P}$.
Thus, via the correspondence $\zeta$ in Lemma~\ref{lem:bij},
we only need to show that there exists
a tagged triangulation $\st$ containing $\alpha$
such that $\alpha$ is a sink in the corresponding associated quiver
and $\varrho(\alpha)$ is the tagged arc obtained by the tagged flip of $\st$ at $\alpha$.

Without loss of generality, suppose that $\alpha$ is not tagged at either end.
We consider two cases:
\medskip

\begin{itemize}
\item[\textbf{I.}]
  If $\alpha=\wideparen{AC}$ is an arc whose endpoints $A,C$ both lie in $M$, we denote $B=\varrho(A)$ and $D=\varrho(C)$.
  The arc $\varrho(\alpha)$ is then the arc connecting $B$ and $D$
that can be obtained as a contraction of the union
\[
    \wideparen{BA}\cup\wideparen{AC}\cup\wideparen{CD},
\]
where $\wideparen{BA}$ and $\wideparen{CD}$ are boundary arc segments, see Figure~\ref{fig:case1}.
  We choose a triangulation $\st$ containing the triangles $\tri ABC$ and $\tri ACD$
  as shown in the left picture of Figure~\ref{fig:case1}.
  \bigskip

\item[\textbf{II.}]
  If $\alpha=\wideparen{AP}$ such that $A\in\M$ and $P\in\P$, we denote $B=\varrho(A)$.
  Choose a triangulation $\st$ containing $\alpha$ and the arc $\wideparen{AB}$ that
  only encloses the puncture $P$.
  Moreover, $\st$ can either contain the ordinary arc $\wideparen{BP}$
  or the tagged arc $\wideparen{AP}$,
  as shown in the middle and right pictures of Figure~\ref{fig:case1}.
\end{itemize}
\medskip

In both cases it is straightforward to see that
$\alpha$ corresponds to a source in the quiver $Q_\hua{P}$ and
the tagged rotation $\varrho(\alpha)$ of $\alpha$ is indeed the tagged arc obtained by the tagged flip of $\st$ at $\alpha$, as required.

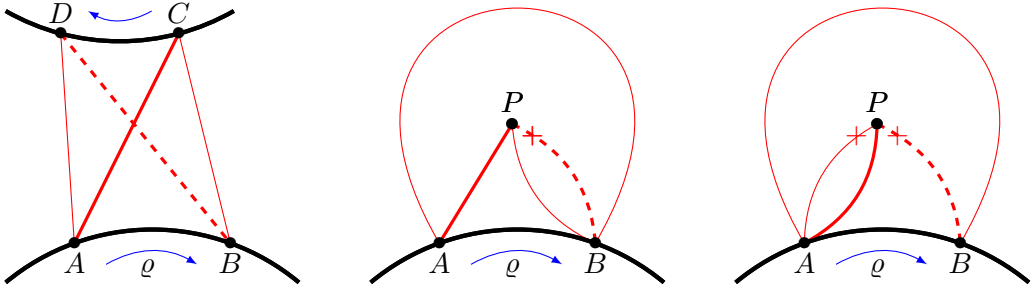
\begin{figure}[t]\centering
\begin{tikzpicture}[scale=.6]
\draw(-1,0)node{};
\draw[ultra thick](0,0)arc(130:110:5)coordinate(C)arc(110:70:5)coordinate(D)arc(70:50:5);
\draw[ultra thick](0,6)arc(-120:-105:5)coordinate(B)arc(-105:-75:5)coordinate(A)arc(-75:-60:5);
\draw[blue,->,>=latex](2.2,.4)to[bend left](4.2,.4);
\draw[blue,<-,>=latex](1.8,6)to[bend right](3.2,6);
\draw(3.1,.3)node{$\varrho$};
\draw[thin,red](D)to(A)to(C)to(B)edge[dashed,very thick](D);\draw[very thick,red](A)to(C);
\draw[ultra thick](0,0)arc(130:110:5)coordinate(C)arc(110:70:5)coordinate(D)arc(70:50:5);
\draw[ultra thick](0,6)arc(-120:-105:5)coordinate(B)arc(-105:-75:5)coordinate(A)arc(-75:-60:5);
\draw(A)node{$\bullet$}node[above]{$C$};\draw(B)node{$\bullet$}node[above]{$D$};
\draw(C)node{$\bullet$}node[below]{$A$};\draw(D)node{$\bullet$}node[below]{$B$};

\draw[ultra thick](0+8,0)arc(130:110:5)coordinate(C)arc(110:70:5)coordinate(D)arc(70:50:5)
    (3.1+8,3.5)coordinate(P)node{$\bullet$}node[above]{$P$}(3.55+8,3.25)node[red]{$+$};
\draw[blue,->,>=latex](2.2+8,.4)to[bend left](4.2+8,.4);
\draw(3.1+8,.3)node{$\varrho$};
\draw[very thick,red](P)to(C);
\draw[thin,red](C)to(P)  (D)   .. controls +(60:8) and +(120:8) ..(C);
\draw[thin,red,bend left] (D)to(P)edge[dashed,very thick](D);
\draw[ultra thick](0+8,0)arc(130:110:5)coordinate(C)arc(110:70:5)coordinate(D)arc(70:50:5)
    (3.1+8,3.5)coordinate(P)node{$\bullet$}node[above]{$P$}(3.55+8,3.25)node[red]{$+$};
\draw(C)node{$\bullet$}node[below]{$A$};\draw(D)node{$\bullet$}node[below]{$B$};

\draw[ultra thick](0+16,0)arc(130:110:5)coordinate(C)arc(110:70:5)coordinate(D)arc(70:50:5)
    (3.1+16,3.5)coordinate(P)node{$\bullet$}node[above]{$P$}
        (3.55+16,3.25)node[red]{$+$}(2.65+16,3.25)node[red]{$+$};
\draw[blue,->,>=latex](2.2+16,.4)to[bend left](4.2+16,.4);
\draw(3.1+16,.3)node{$\varrho$};
\draw[thin,red](D)   .. controls +(60:8) and +(120:8) ..(C);
\draw[thin,red,bend left] (C)to(P)edge[very thick](C);
\draw[thin,red,bend left] (P)edge[dashed,very thick](D);
\draw[ultra thick](0+16,0)arc(130:110:5)coordinate(C)arc(110:70:5)coordinate(D)arc(70:50:5)
    (3.1+16,3.5)coordinate(P)node{$\bullet$}node[above]{$P$};
\draw(C)node{$\bullet$}node[below]{$A$};\draw(D)node{$\bullet$}node[below]{$B$};
\end{tikzpicture}
\caption{The flip/tagged rotation of certain types of arcs}
\label{fig:case1}
\end{figure}
\end{proof}

We consider now a loop $\alpha$ whose endpoint $A$ is a marked point in $M$, and such that $\alpha$ is not contractible in the surface $\surf$, that is, $\alpha$ is non-trivial in the homology group even if we
glue a disc to each boundary component and ignore all the marked points. This  implies that $\alpha$ is a non-separating loop with endpoint in $\M$.
In this case the surface has genus at least one, and we can realize the arc $\alpha$ as the arc labeled by $1$ in the upper left picture in Figure~\ref{fig:mutation}. The blue triangle depicts the rest of the surface, in particular,  the blue region might contain
punctures, boundary components and handle bodies.
It is also possible that the arcs labeled $4$ and $5$ coincide, in which case
$\surf$ is just a torus with one boundary component and one marked point.
The rotation $\varrho$ yields the next point $\varrho(A)=B$ in $\M$ on the same boundary component. We do not insist that $A$ and $B$ are different points, as for in the  case of a torus with one boundary component and one marked point the points $A$ and $B$ coincide.

\begin{lemma}\label{lem:T-rotation2}
Let $\surf$ be a surface of genus at least one, and let $\alpha$ be a loop with endpoint in $\M$ such that $\alpha$ is not contractible in the surface $\surf$.
Then we have $\zeta(\varrho(\alpha))=\zeta(\alpha)[1]$.
\end{lemma}
\begin{proof}
We choose a triangulation of $\surf$ containing the arc $\alpha$ as the arc labeled $1$ in Figure~\ref{fig:mutation}, with a presentation of one handle body given by the arcs labeled $2$ and $3$. The arc  $\varrho(\alpha)$\ is then the arc labeled with $9$.
In order to show the statement of the lemma, let $X_i$ denote the rigid object  in the cluster category which is obtained under the map $\zeta$
from the arc labeled by $i$.
We have a mutation sequence $\mu_{321}=\mu_3\circ\mu_2\circ\mu_1$ as shown in Figure~\ref{fig:mutation},
starting from the upper left picture to the lower left one,
then the middle one and finally the right one.
The mutation sequence of the corresponding quivers is shown in Figure~\ref{fig:quiver}.
The corresponding exchange triangles of those three mutations are
\begin{gather*}
  X_1 \xrightarrow{} X_2 \xrightarrow{} X_7 \xrightarrow{} X_1[1],\\
  X_2 \xrightarrow{} X_3\oplus X_7 \xrightarrow{} X_8 \xrightarrow{} X_2[1],\\
  X_3 \xrightarrow{} X_8 \xrightarrow{} X_9 \xrightarrow{} X_3[1].
\end{gather*}
The first two triangles fit into the following diagram of triangles
\begin{equation}
\xymatrix@C=2.5pc@R=2.5pc{
  & X_2 \ar@{=}[r]\ar[d] & X_2\ar[d]\\
  X_3\ar@{=}[d] \ar[r] & X_3\oplus X_7\ar[d] \ar[r] & X_7\ar[d] \ar[r] & X_3[1]\ar@{=}[d]\\
  X_3 \ar[r] & X_8\ar[d] \ar@{.>}[r] & X_1[1]\ar[d] \ar@{.>}[r] & X_3[1] \\
  & X_2[1] \ar@{=}[r] & X_2[1]
}\label{eq:comm.diag}\end{equation}
It is straightforward to see that the leftmost and the top square are commutative.
For instance, the map from $X_2$ to $X_7$ in the top square is
the irreducible map corresponding to the arrow $7\to2$
in the second quiver of Figure~\ref{fig:quiver}.
By the Octahedral Axiom, the third line (with dotted arrows) in \eqref{eq:comm.diag} is a triangle,
which coincides with the third exchange relation.
In other words, $X_9=X_1[1]$, as was to show.
\end{proof}

\begin{figure}[h]\centering
  \begin{tikzpicture}[scale=.7]
  \draw (0-0.4,6)    node{$_A$};
  \draw (4.4,10.0)    node{$_A$};
  \draw (1.5,8.3)    node{$_B$};
  \draw (0-.3,8)    node{$_3$};
  \draw (4+.3,8)    node{$_3$};
  \draw (2,6-.3)    node{$_2$};
  \draw (2,10+.3)    node{$_2$};
  \draw (3-.7,9-.1)    node[blue]{$_4$};
  \draw (3-.5,7+.5)    node[blue]{$_5$};
  \draw (1+.4,3+.1)    node{$_7$};
  \draw (8+.4,5+.5)    node{$_8$};
  \draw (16.5,7)    node{$_9$};
  \draw (1+.2,9-.2)    node{$_1$};
    \draw(0,0+6)coordinate(a1) (0,4+6)coordinate(a2) (4,0+6)coordinate(a3) (4,4+6)coordinate(a4)
        (2-.2,2+6+.2)coordinate(a5);
        \draw[cyan!15,fill=cyan!15](a1)to(a5)to(a4)to[bend left=10](a1);
    \draw[very thick](a1)to(a5);
    \draw[red,thin](a1)to(a2)to(a4)to(a3)to(a1)to[bend right=10](a4)to(a5);
    \foreach \j in {1,...,5}  {\draw[red,thin] (a\j)node{$\bullet$};}
    \path (a2)--(a1) coordinate[pos=.5] (m1);\path (a2)--(a4) coordinate[pos=.55] (m2);
    \path (a3)--(a1) coordinate[pos=.45] (m4);\path (a3)--(a4) coordinate[pos=.5] (m3);
    \draw[red,thin,dotted](a2)to(a1)edge[->,>=latex](m1);
    \draw[red,thin,dotted](a4)to(a3)edge[->,>=latex](m3);
    \draw[red,thin,dotted](a4)to(a2)edge[->>,>=latex](m2);
    \draw[red,thin,dotted](a3)to(a1)edge[->>,>=latex](m4);
\draw[red,ultra thick](a1)to[bend left=30](a4);

    \draw(0,0)coordinate(a1) (0,4)coordinate(a2) (4,0)coordinate(a3) (4,4)coordinate(a4)
        (2-.2,2+.2)coordinate(a5);
        \draw[cyan!15,fill=cyan!15](a1)to(a5)to(a4)to[bend left=10](a1);
    \draw[very thick](a1)to(a5);
    \draw[red,thin](a1)to(a2)to(a4)to(a3)to(a1)to[bend right=10](a4)to(a5);
    \foreach \j in {1,...,5}  {\draw[red,thin] (a\j)node{$\bullet$};}
    \path (a2)--(a1) coordinate[pos=.5] (m1);\path (a2)--(a4) coordinate[pos=.55] (m2);
    \path (a3)--(a1) coordinate[pos=.45] (m4);\path (a3)--(a4) coordinate[pos=.5] (m3);
    \draw[red,thin,dotted](a2)to(a1)edge[->,>=latex](m1);
    \draw[red,thin,dotted](a4)to(a3)edge[->,>=latex](m3);
    \draw[red,thin,dotted](a4)to(a2)edge[->>,>=latex](m2);
    \draw[red,thin,dotted](a3)to(a1)edge[->>,>=latex](m4);
\draw[ultra thick, red](a5)to(a2);

    \draw(6,1)coordinate(a1) (6,5)coordinate(a2) (10,1)coordinate(a3) (10,5)coordinate(a4)
        (8-.2,3+.2)coordinate(a5);
        \draw[cyan!15,fill=cyan!15](a1)to(a5)to(a4)to[bend left=10](a1);
    \draw[very thick](a1)to(a5);
    \draw[red,thin](a1)to(a2) (a4)to(a3) (a1)to[bend right=10](a4)to(a5);
    \foreach \j in {1,...,5}  {\draw[red,thin] (a\j)node{$\bullet$};}
    \path (a2)--(a1) coordinate[pos=.5] (m1);\path (a2)--(a4) coordinate[pos=.55] (m2);
    \path (a3)--(a1) coordinate[pos=.45] (m4);\path (a3)--(a4) coordinate[pos=.5] (m3);
    \draw[red,thin,,dotted](a2)to(a1)edge[->,>=latex](m1);
    \draw[red,thin,,dotted](a4)to(a3)edge[->,>=latex](m3);
    \draw[Emerald,thin,dotted](a4)to(a2)edge[->>,>=latex](m2);
    \draw[Emerald,thin,dotted](a3)to(a1)edge[->>,>=latex](m4);
\draw[ultra thick, red](10,9)to(a5);\draw[red,thin](a5)to(a2);
    \draw(6,1+4)coordinate(a1) (6,5+4)coordinate(a2) (10,1+4)coordinate(a3) (10,5+4)coordinate(a4)
        (8-.2,3+4+.2)coordinate(a5);
        \draw[cyan!15,fill=cyan!15](a1)to(a5)to(a4)to[bend left=10](a1);
    \draw[very thick](a1)to(a5);
    \draw[red,thin](a1)to(a2) (a4)to(a3) (a1)to[bend right=10](a4)to(a5);
    \foreach \j in {1,...,5}  {\draw[red,thin] (a\j)node{$\bullet$};}
    \path (a2)--(a1) coordinate[pos=.5] (m1);\path (a2)--(a4) coordinate[pos=.55] (m2);
    \path (a3)--(a1) coordinate[pos=.45] (m4);\path (a3)--(a4) coordinate[pos=.5] (m3);
    \draw[red,thin,,dotted](a2)to(a1)edge[->,>=latex](m1);
    \draw[red,thin,,dotted](a4)to(a3)edge[->,>=latex](m3);
    \draw[Emerald,thin,dotted](a4)to(a2)edge[->>,>=latex](m2);
    \draw[Emerald,thin,dotted](a3)to(a1)edge[->>,>=latex](m4);
\draw[red,thin](8.666667,9)to(a5); \draw[red,thin](a5)to(a2);
\draw[red,thin](8.666667,1)to(10,5);


    \draw(12,1)coordinate(a1) (12,5)coordinate(a2) (16,1)coordinate(a3) (16,5)coordinate(a4)
        (14-.2,3+.2)coordinate(a5);
        \draw[cyan!15,fill=cyan!15](a1)to(a5)to(a4)to[bend left=10](a1);
    \draw[very thick](a1)to(a5);
    \foreach \j in {1,...,5}  {\draw[red,thin] (a\j)node{$\bullet$};}
    \path (a2)--(a1) coordinate[pos=.5] (m1);\path (a2)--(a4) coordinate[pos=.55] (m2);
    \path (a3)--(a1) coordinate[pos=.45] (m4);\path (a3)--(a4) coordinate[pos=.5] (m3);
    \draw[Emerald,thin,dotted](a1)edge[->,>=latex](m1);
    \draw[Emerald,thin,dotted](a3)edge[->,>=latex](m3);
    \draw[Emerald,thin,dotted](a2)edge[->>,>=latex](m2);
    \draw[Emerald,thin,dotted](a1)edge[->>,>=latex](m4);
    \draw[Emerald,thin,dotted](a2)to(m1)(a4)to(m3)(a4)to(m2)(a3)to(m4);
\draw[red,thin](16,9)to(a5)to(a2)  (a1)to[bend right=10](a4)to(a5);
    \draw(12,1+4)coordinate(a1) (12,5+4)coordinate(a2) (16,1+4)coordinate(a3) (16,5+4)coordinate(a4)
        (14-.2,3+4+.2)coordinate(a5);
        \draw[cyan!15,fill=cyan!15](a1)to(a5)to(a4)to[bend left=10](a1);
    \draw[very thick](a1)to(a5);
    \foreach \j in {1,...,5}  {\draw[red,thin] (a\j)node{$\bullet$};}
    \path (a2)--(a1) coordinate[pos=.5] (m1);\path (a2)--(a4) coordinate[pos=.55] (m2);
    \path (a3)--(a1) coordinate[pos=.45] (m4);\path (a3)--(a4) coordinate[pos=.5] (m3);
    \draw[Emerald,thin,dotted](a1)edge[->,>=latex](m1);
    \draw[Emerald,thin,dotted](a2)edge[->>,>=latex](m2);
    \draw[Emerald,thin,dotted](a2)to(m1)(a4)to(m2);
\draw[red,thin](14.666667,9)to(a5)to(a2);
\draw[red,thin](14.666667,1)to(16,5)  (a1)to[bend right=10](a4)to(a5);

    \draw(12+4,1)coordinate(a1) (12+4,5)coordinate(a2) (16+4,1)coordinate(a3)
        (16+4,5)coordinate(a4)  (14+4-.2,3+.2)coordinate(a5);
        \draw[cyan!15,fill=cyan!15](a1)to(a5)to(a4)to[bend left=10](a1);
    \draw[very thick](a1)to(a5);
    \foreach \j in {1,...,5}  {\draw[red,thin] (a\j)node{$\bullet$};}
    \path (a2)--(a1) coordinate[pos=.5] (m1);\path (a2)--(a4) coordinate[pos=.55] (m2);
    \path (a3)--(a1) coordinate[pos=.45] (m4);\path (a3)--(a4) coordinate[pos=.5] (m3);
    \draw[Emerald,thin,dotted](a3)edge[->,>=latex](m3);
    \draw[Emerald,thin,dotted](a1)edge[->>,>=latex](m4);
    \draw[Emerald,thin,dotted](a4)to(m3)(a3)to(m4);
\draw[red,thin](16+4,9)to(a5)to(a2)  (a1)to[bend right=10](a4)to(a5);
    \draw(12+4,1+4)coordinate(a1) (12+4,5+4)coordinate(a2)
        (16+4,1+4)coordinate(a3) (16+4,5+4)coordinate(a4)
        (14+4-.2,3+4+.2)coordinate(a5);
        \draw[cyan!15,fill=cyan!15](a1)to(a5)to(a4)to[bend left=10](a1);
    \draw[very thick](a1)to(a5);
    \foreach \j in {1,...,5}  {\draw[red,thin] (a\j)node{$\bullet$};}
    \path (a2)--(a1) coordinate[pos=.5] (m1);\path (a2)--(a4) coordinate[pos=.55] (m2);
    \path (a3)--(a1) coordinate[pos=.45] (m4);\path (a3)--(a4) coordinate[pos=.5] (m3);
    \draw[Emerald,thin,dotted](a1)edge[->,>=latex](m1);
    \draw[Emerald,thin,dotted](a3)edge[->,>=latex](m3);
    \draw[Emerald,thin,dotted](a2)edge[->>,>=latex](m2);
    \draw[Emerald,thin,dotted](a1)edge[->>,>=latex](m4);
    \draw[Emerald,thin,dotted](a2)to(m1)(a4)to(m3)(a4)to(m2)(a3)to(m4);
\draw[red,thin](14+4.666667,9)to(a5)to(a2);
\draw[red,thin](14+4.666667,1)to(16+4,5)  (a1)to[bend right=10](a4)to(a5);

\draw[ultra thick, red](14-.2,3+.2)  to (15.3333-.2,5) to (16,6+.1)
                  .. controls +(50:1) and +(180:1) ..  (18-.2,7+.2);
\draw[red,thin]
    (18-.2,3+.2)to(19.3333-.2,5)     (14-.2,7+.2)to(15.3333-.2,9)     (18-.2,7+.2)to(19.3333-.2,9)
    (19.3333-.2,5)to(20,6+.1)     (15.3333-.2,5-4)to(16,6-4+.1)     (19.3333-.2,5-4)to(20,6-4+.1);
\draw[red,thin]  (16-4,6+.1).. controls +(50:1) and +(180:1) ..  (18-4-.2,7+.2);
\draw[red,thin]  (16,6-4+.1).. controls +(50:1) and +(180:1) ..  (18-.2,7-4+.2);
\draw[red,thin]  (16-4,6-4+.1).. controls +(50:1) and +(180:1) ..  (18-4-.2,7-4+.2);
  \end{tikzpicture}
\caption{The mutation sequence $\mu_{321}$ on a higher genus surface}
\label{fig:mutation}

\begin{tikzpicture}[scale=.9,rotate=-45]
\draw(0,0)node(a3){$3$}
     (0,2)node[blue](a5){$5$}(0,3)node[above,white]{$x$}
     (2,0)node(a1){$1$}node{$\bigcirc$}
     (2,2)node(a2){$2$}
     (3,-1)node[blue](a4){$4$};
\draw[->,>=latex] (a2)to(a1);\draw[->,>=latex] (a2)to(a5);\draw[->,>=latex] (a1)to(a4);
\draw[->,>=latex] (a1)to(a3);\draw[->,>=latex] (a5)to(a3);
\draw[->,>=latex] (a3)to[bend right=10](a2);
\draw[->,>=latex] (a3)to[bend left=10](a2);
\end{tikzpicture}
\;
\begin{tikzpicture}[scale=.9,rotate=-45]
\draw(0,0)node(a3){$3$}
     (0,2)node[blue](a5){$5$}
     (2,0)node(a1){$7$}
     (2,2)node(a2){$2$}node{$\bigcirc$}
     (3,-1)node[blue](a4){$4$};
\draw[<-,>=latex] (a2)to(a1);\draw[->,>=latex] (a2)to(a5);\draw[<-,>=latex] (a1)to(a4);
\draw[<-,>=latex] (a1)to(a3);\draw[->,>=latex] (a5)to(a3);\draw[<-,>=latex] (a4)to(a2);
\draw[->,>=latex] (a3)to(a2);
\end{tikzpicture}
\;
\begin{tikzpicture}[scale=.9,rotate=-45]
\draw(0,0)node(a3){$3$}node{$\bigcirc$}
     (0,2)node[blue](a5){$5$}
     (2,0)node(a7){$7$}
     (2,2)node(a8){$8$}
     (3,-1)node[blue](a4){$4$};
\draw[->,>=latex] (a5)to(a8);\draw[->,>=latex] (a8)to(a3);
\draw[->,>=latex] (a8)to(a7);\draw[->,>=latex] (a3)to(a4);
\draw[->,>=latex] (a3)to(a7);\draw[->,>=latex] (a4)to(a8);
\draw[->,>=latex] (a7)to(a5);
\end{tikzpicture}
\;
\begin{tikzpicture}[scale=.9,rotate=-45]
\draw(0,0)node(a9){$9$}
     (0,2)node[blue](a5){$5$}
     (2,0)node(a7){$7$}
     (2,2)node(a8){$8$}
     (3,-1)node[blue](a4){$4$};
\draw[->,>=latex] (a5)to(a8);\draw[<-,>=latex] (a8)to(a9);
\draw[->,>=latex] (a8)to[bend right=10](a7);
\draw[->,>=latex] (a8)to[bend left=10](a7);
\draw[<-,>=latex] (a9)to(a4);
\draw[<-,>=latex] (a9)to(a7);\draw[->,>=latex] (a7)to(a5);
\end{tikzpicture}
\caption{The mutation sequence $\mu_{321}$ of the quiver}
\label{fig:quiver}
\end{figure}

\begin{remark}\label{rem:T-rotation}
Presumably, it is possible to use the same method as in Lemma~\ref{lem:T-rotation2}
for all cases of tagged arcs.
However, this becomes complicated  when the arc connects two punctures.
Instead, we use the tagged mapping class group to give a more conceptual proof.
\end{remark}

\begin{theorem}\label{thm:T-rotation}
The tagged rotation $\varrho\in\TM(\surf)$ on $\TA$ becomes the shift $[1]$ on $\obj$, i.e.
$\iota_\surf(\varrho)=[1]$.
\end{theorem}
\begin{proof}
We only need to show that there exists a tagged triangulation $\st$
consisting of the types of arcs considered
in Lemma~\ref{lem:T-rotation} and Lemma~\ref{lem:T-rotation2}.
Since the action on $\obj$ of an element in $\Aut_0\C{}{\surf}$ is determined by
the action on a particular cluster tilting set,
the theorem then follows by Lemma~\ref{lem:MCG}.

\begin{figure}[t]\centering
\begin{tikzpicture}[scale=.5]
  \foreach \j in {1,2,3,5,6,7}{  \draw[thick,red] (22.5-90:6)--(22.5+45*\j-90:6);}
  \draw[thick,red] (22.5-90:6)to[bend left=15](22.5+45*4-90:6);
  \draw[thick,red] (22.5-90:6)to[bend right=15](22.5+45*4-90:6);
  \foreach \j in {1,2,5,6}{
        \draw[thick,red](22.5-45*\j:6)edge[->,>=latex](22.5-45*\j+45:6)node{$\bullet$};}
  \foreach \j in {3,4,7,8}{
        \draw[thick,red](22.5-45*\j:6)edge[<-,>=latex](22.5-45*\j+45:6)node{$\bullet$};}

\draw[very thick, fill=gray!11] (22.5-90:6)
    .. controls +(90+10:5) and +(90+45-10:5) ..(22.5-90:6);

\draw(-45-45*1:6.5)node{$a_1$};\draw(-45-45*3:6.5)node{$a^{-1}_1$};
\draw(-45-45*2:6.5)node{$b_1$};\draw(-45-45*4:6.5)node{$b^{-1}_1$};
\draw(-45-45*5:6.5)node{$a_2$};\draw(-45-45*7:6.5)node{$a^{-1}_2$};
\draw(-45-45*6:6.5)node{$b_2$};\draw(-45-45*8:6.5)node{$b^{-1}_2$};
\end{tikzpicture}
\caption{The basic case \textbf{i}: $4g$-gon presentation of a genus $g=2$ surface}
\label{fig:9}

\begin{tikzpicture}[scale=.3]
  \draw[very thick](0,8)node{};
  \draw[very thick](0,0)circle(6);  \draw[very thick, fill=gray!11](0,0)circle(2);
  \draw[thick,red] (0,-6).. controls +(135:3) and +(180:6) ..(0,2)
    .. controls +(0:6) and +(45:3) ..(0,-6);
  \draw[very thick](0,-6)node{$\bullet$}(0,2)node{$\bullet$};
\end{tikzpicture}
\begin{tikzpicture}[scale=.3]
\draw[thick,red](0,-1)to[bend left=60](0,-6);
\draw[thick,red](0,-1)to[bend right=60](0,-6);

  \draw[very thick](0,7)node{};
  \draw[very thick](0,0)circle(6);
  \draw[thick,red] (0,-6).. controls +(155:3) and +(180:6) ..(0,2)
    .. controls +(0:6) and +(25:3) ..(0,-6);
  \draw[very thick](0,-6)node{$\bullet$}(0,2)node{$\bullet$}(0,-1)node{$\bullet$}
    (0.5,-1.4)node[red]{+};
\end{tikzpicture}
\begin{tikzpicture}[scale=.3]
  \draw[very thick](0,7)node{};
  \draw[very thick](0,0)circle(6);
  \draw[red, thick] (0,6)--(0,0)--(0,-6);
  \draw[very thick](0,-6)node{$\bullet$}(0,0)node{$\bullet$}(0,6)node{$\bullet$};
\end{tikzpicture}
\begin{tikzpicture}[scale=.33,rotate=45]
  \draw[red, thick] (0,6)(0,-6);  \draw[red, thick] (0,6)--(0,-6);
  \draw[very thick](0,6)node{$\bullet$}--(6,0)node{$\bullet$}--
        (0,-6)node{$\bullet$}--(-6,0)node{$\bullet$}--cycle;
\end{tikzpicture}
\caption{The basic cases \textbf{ii}, \textbf{iii}, \textbf{iv}, \textbf{v}}
\label{fig:basic}
\end{figure}

Any marked surface falls into at least one of the following cases:
\begin{itemize}
  \item[\textbf{I.}] $g>0$,
  \item[\textbf{II.}] $g=0$, $b\geq2$,
  \item[\textbf{III.}] $g=0$, $b=1$, $p\geq2$,
  \item[\textbf{IV.}] $g=0$, $b=1$, $p=1$, $m\geq2$,
  \item[\textbf{V.}] $g=0$, $b=1$, $p=0$ and $m\geq4$.
\end{itemize}
We use induction on the number $n=6g+3p+3b+m-6$ defined in equation \ref{eq:n}, starting from one of the following basic cases which represent the minimal possible $n$ in each of the five cases I to V above:
\begin{itemize}
  \item[\textbf{i.}] $g>0$, $b=m=1$ and $p=0$,
  \item[\textbf{ii.}] $g=0$, $b=m=2$ and $p=0$,
  \item[\textbf{iii.}] $g=0$, $b=m=1$ and $p=2$,
  \item[\textbf{iv.}] $g=0$, $b=p=1$ and $m=2$,
  \item[\textbf{v.}] $g=0$, $b=1$, $p=0$ and $m=4$.
\end{itemize}
Any marked surface can be obtained from
one of the corresponding basic cases by adding
\begin{itemize}
\item[(a)] a marked point to an existing boundary component,
\item[(b)] a boundary component with one marked point, or
\item[(c)] a puncture.
\end{itemize}
As start of the induction, we first show that in each of  the basic cases \textbf{i} to \textbf{v} there exists a tagged triangulation
consisting of the two types of arcs considered in Lemma~\ref{lem:T-rotation} and Lemma~\ref{lem:T-rotation2}.
For the basic case \textbf{i}, Figure~\ref{fig:9} shows for $g=2$  a tagged triangulation where every arc is a non-contractible loop with endpoint in $\M$. Thus, all arcs are of the type considered  in Lemma~\ref{lem:T-rotation2}, and an analogous triangulation can be defined for arbitrary $g \ge 1$.
For the basic cases \textbf{ii, iii, iv} and \textbf{v}, we show in Figure~\ref{fig:basic} a triangulation where each arc has two different endpoints, at least one of them belonging to $\M$. Thus, these triangulations
consist of the type of arc considered in Lemma~\ref{lem:T-rotation}.
Note that in each of these basic cases, there is a triangle with a boundary arc.

We now proceed to the inductive step. Figure~\ref{fig:add} shows that in the each of the three cases (a), (b) and (c) of adding,
we can modify the tagged triangulation $\st$ of the old surface
(within the triangle with a boundary arc)
to get a tagged triangulation $\st'$ for the new surface,
such that $\st'$ only consists of the types of arcs considered  in Lemma~\ref{lem:T-rotation} and Lemma~\ref{lem:T-rotation2}, which completes the proof.

More precisely, in case (a) of adding a marked point C to an existing boundary component, we illustrate in the left picture of Figure~\ref{fig:add} how to add a new arc of the type considered  in Lemma~\ref{lem:T-rotation}.

In case (b) of adding a boundary component with one marked point C, we show in the middle picture of Figure~\ref{fig:add} how to add the component into a triangle with a boundary edge such that all new arcs are of the type considered  in Lemma~\ref{lem:T-rotation}.

Finally, in case (c) of adding a puncture D into a triangle ABC with a boundary edge AB, we depict the situation in the right picture of Figure~\ref{fig:add}. The added arcs connecting A and the puncture D are of the type considered  in Lemma~\ref{lem:T-rotation}, as well as the new arc connecting A and C in case these are different points. The case when A and C coincide within a triangle with  boundary edge AB occurs in our construction only when one starts with type \textbf{i}, thus in this case the new arc is  of the type considered  in Lemma~\ref{lem:T-rotation2}.

\begin{figure}[h]\centering
\begin{tikzpicture}[scale=.5]
 \draw (0.5,0.5)    node{$_C$};

  \draw[thick,red] (-3,0)--(0,7)--(3,0) (0,8);
  \draw[very thick](-3,0)node{$\bullet$}--(3,0)node{$\bullet$}(0,7)node{$\bullet$};

  \draw[thick, Emerald](0,0)--(0,7);
  \draw(0,0)node[white]{$\bullet$}node[thick, Emerald]{$\bullet$};
  \draw(0,7)node{$\bullet$};
\end{tikzpicture}
\begin{tikzpicture}[scale=.5]
  \draw[thick,red] (-3,0)--(0,7)--(3,0);
 \draw (0,0.7)    node{$_C$};
  \draw[thick, Emerald, fill=gray!11](0,2)circle(.7);
  \draw[thick, Emerald](3,0)--(0,1.3)--(-3,0);

  \draw[thick, Emerald] (0,1.3).. controls +(-180:2) and +(-90-15:4) ..(0,7);
  \draw[thick, Emerald] (0,1.3).. controls +(0:2) and +(-90+15:4) ..(0,7);

  \draw[very thick](-3,0)node{$\bullet$}--(3,0)node{$\bullet$}(0,7)node{$\bullet$};
  \draw(0,1.3)node[white]{$\bullet$}(0,1.3)node[thick, Emerald]{$\bullet$};
\end{tikzpicture}
\begin{tikzpicture}[scale=.5]
 \draw (-3.4,0.5)    node{$_A$};
 \draw (3.4,0.5)    node{$_B$};
 \draw (0,3.5)    node{$_D$};
\draw (0.5,7.5)    node{$_C$};
   \draw[thick,red] (-3,0)--(0,7)--(3,0) (0,8.5);

  \draw[thick, Emerald] (-3,0)to[bend left=15](0,3);
  \draw[thick, Emerald] (-3,0)to[bend left=-15](0,3);
  \draw[thick, Emerald] (-3,0).. controls +(10:4) and +(-90+15:5) ..(0,7);
  \draw[thick, Emerald] (0-.5,3-.3) node {+};

  \draw(0,3)node[white]{$\bullet$}(0,3)node[thick, Emerald]{$\bullet$};
  \draw[very thick](-3,0)node{$\bullet$}--(3,0)node{$\bullet$}(0,7)node{$\bullet$};
\end{tikzpicture}
\caption{Three cases of adding (in green)}
\label{fig:add}
\end{figure}
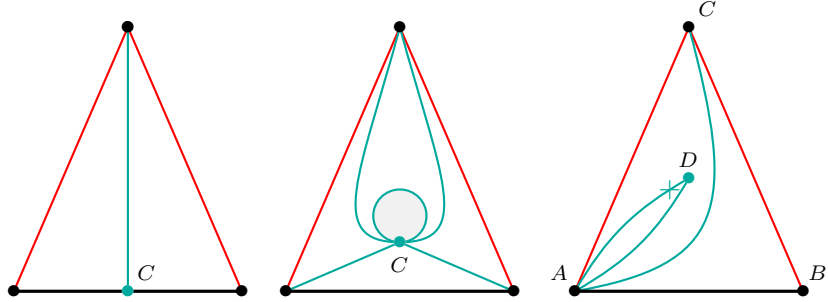
\end{proof}

\section{Applications}
\subsection{On the Jacobian algebras}
As discussed earlier, a tagged triangulation $\T$ of a marked surface $\surf$ with non-empty boundary defines  a Jacobi-finite quiver $Q$ with potential $W$.
We denote the corresponding Jacobi algebra $\J(Q,W)$ by $\J(\T)$, and $\mod \J(\T)$ denotes the category of finitely generated $\J(\T)-$modules.
Moreover, we denote by $\Add\zeta(\T) $ the additive subcategory of $\hua{C}(\surf)$ generated by the object $\zeta(\T) $, that is, all direct sums of summands of $\zeta(\T) $.

\begin{lemma}\label{lem:quotient}\cite{A,KR}
There is a canonical equivalence
\[\theta_\T\colon\hua{C}(\surf)/\Add\zeta(\T) \to \mod \J(\T).\]
\end{lemma}

By the proof in \cite[Section~3.5]{KR},
the AR-translation on $\J(\T)$ is induced from the
AR-translation $\tau$ (or the shift $[1]$) of the cluster category $\hua{C}(\surf)$.
Thus we have the following corollary.

\begin{corollary}\label{cor:AR}
Under the bijection $\theta_\T$ in Lemma~\ref{lem:quotient},
the tagged rotation $\varrho$ on $\TA$ induces the AR-translation $\tau$
on the objects in $\mod \J(\T)$ that are the images of the reachable rigid objects.
\end{corollary}

\subsection{Shifts and Seidel-Thomas braid group}
Let $\Gamma=\Gamma(Q,W)$ be the Ginzburg dg algebra of any quiver with potential.
It is known that $\D_{fd}(\Gamma)$ is a Calabi-Yau-$3$ category,
which admits a standard heart $\zero$ generated by simple $\Gamma$-modules $S_e$, for $e\in Q_0$,
each of which is a $3$-spherical object.
Recall that an object $S$ is \emph{$N$-spherical} when
$\Hom^{\bullet}(S, S)=\k \oplus \k[-N]$.
Moreover, we recall (e.g. \cite{ST})
a distinguished family of auto-equivalences of $\D_{fd}(\Gamma)$.

\begin{definition}
\label{def:twist}
The \emph{twist functor $\phi$ of a spherical object} $S$
is defined by
\begin{gather}
    \phi_S(X)=\Cone\left(X\to S\otimes\Hom^\bullet(X,S)^\vee \right)[-1]
\end{gather}
with inverse
\begin{gather}
    \phi_S^{-1}(X)=\Cone\left(S\otimes\Hom^\bullet(S,X)\to X\right)
\end{gather}
Her we denote the graded dual of a graded $\k$-vector space
$V=\oplus_{j\in\ZZ} V_j[j]$
by $V^\vee=\bigoplus_{j\in\ZZ} V_j^*[-j]$.
The \emph{Seidel-Thomas braid group}, denoted by $\Br\Gamma$,
is the subgroup of $\Aut\D_{fd}(\Gamma)$ generated by
the twist functors of the simples in $\Sim\zero$.
\end{definition}
For each simple object $S$ in a reachable heart $\mathcal H$ in $\D_{fd}(\Gamma)$  there are two different simple tilts that can be applied to $\mathcal H$, and the resulting tilted hearts (there forward and the backward tilt with respect to $S$) are related to each other by a spherical twist, see \cite{K10}. On the level of the cluster category the two different tilts coincide, thus the cluster exchange graph can be identified with the exchange graph of reachable hearts modulo the braid group.
We apply this to surface triangulations: Given a tagged triangulation $\T$ of $\surf$, we recall
from section 2 that the category $\D(\surf)= \D_{fd}(\Gamma)$ is independent (up to derived equivalence) of the triangulation $\T$ used to define the dg algebra $\Gamma$. We accordingly denote the Seidel-Thomas braid group of $\D(\surf)$ by $\Br \surf$ and
we also denote by $\EGp(\D(\surf))$ the principal component of the exchange graph of hearts in $\D(\surf)$,
which consists of all hearts that can be obtained
through iterated (simple) tilts from the standard heart $\zero$.
Let $\CEG{}{\surf}$ be the \emph{oriented exchange graph} of $\C{}{\surf}$,
which can be obtained from $\CEG{*}{\surf}$ by replacing each unoriented edge with a 2-cycle.
By \cite{A} (cf. \cite[Theorem~8.6]{KQ}),
there is a map
\begin{gather}\label{eq:ups}
    \upsilon:\EGp(\D(\surf)) \to \CEG{}{\surf}
\end{gather}
which induces (cf. \cite{K6} for the general case and \cite{KQ} for the acyclic case) an isomorphism
\[
    \CEG{}{\surf} \cong \EGp(\D(\surf))/\Br \surf.
\]
More precisely, the map $\upsilon$ is defined as follows.
Let $\h$ be a heart in $\EGp(\D(\surf))$ which corresponds to
a t-structure $\hua{P}$ in $\D_{fd}(\Gamma)$.
Lift $\hua{P}$ to a t-structure in $\per(\Gamma)$ via the inclusion in \eqref{eq:ses}.
Such a t-structure corresponds to a silting object of $\per(\Gamma)$.
The image of this silting object under the functor $\per(\Gamma) \to \C{}{\Gamma} $
is defined to be $\upsilon(\h)$.
Furthermore, it is not hard to show that
$\upsilon$ commutes with the shift functor, i.e.
\[\upsilon(\h[1])=\upsilon(\h)[1].\]

\begin{theorem}\label{thm:shift}
Let $\surf$ be  a marked surface
which is not a polygon or a once-punctured polygon.
Then the intersection of $\Br\surf$ with the subgroup  $\ZZ[1]$ of shifts in {\rm Aut} $\D(\surf)$ is trivial.
\end{theorem}

\begin{proof}
Suppose that $[m]\in\Br \surf$ for some integer $m$.
Since the induced action of the braid group on $\CEG{}{\surf}$ is trivial,
$[m]$ acts trivially on $\CEG{}{\surf}$, or $[m]=\id$ in $\Aut_0\C{}{\surf}$.
On the other hand, we claim that there always exists an arc $\gamma$ such that the orbit $\{\varrho^z(\gamma) | z \in \ZZ \}$ has infinite order:
If $\surf$ admits two different boundary components, we can choose any arc $\gamma$ connecting the two components. Likewise, if $\surf$ has more than one puncture, for any arc $\gamma$ connecting one puncture to a marked point on the boundary the set  $\{\varrho^z(\gamma) | z \in \ZZ \}$ contains infinitely many non-isotopic curves.

Thus $\surf$ can have at most one boundary component and one puncture.
Moreover, the genus of $\surf$ must be zero since
the Dehn twist $D$ of a non-contractible closed curve in a higher genus surface has infinite order.
More precisely, for any arc $\gamma$ whose endpoints coincide on the one marked point on the boundary and such that $\gamma$ is non-trivial viewed as an element in  the fundamental group of the surface obtained from $\surf$ by gluing a disc to the boundary component, the orbit of $\gamma$ under $D$ is infinite.
\end{proof}

Let $(Q,W)$ be a quiver with potential
associated to some triangulation of a marked surface $\surf$.
Then $\surf$ is not a polygon or a once-punctured polygon if and only if
$(Q,W)$ is not mutation-equivalent to a quiver of type A or D.
The theorem can thus be rephrased by saying the Thomas-Seidel braid group intersects trivially with the shifts in {\rm Aut} $\D(\surf)$ except in the cluster types A or D.

\subsection{Center of the braid group}

In this subsection, we discuss the centers of the braid groups
of type A or D.
Let $Q$ be a quiver of type A or D, and denote by $\Br_Q$ the \emph{braid group of } $Q$ which is defined by
 the  generating set
$\mathbf{b} = \{ b_i \, | \,i \in Q_0 \}$ and relations
\[b_i b_j = b_j b_i\]
 if there is no arrow between the vertices $i$ and $j$, and
\[b_i b_j b_i= b_j b_i b_j\]
 in case there is an arrow between $i$ and $j$.
Recall that
the quasi-center of $\Br_Q$ is the subgroup of elements
$\tri(Q)$ in $\Br_Q$ satisfying $\tri(Q)\cdot \mathbf{b}\cdot \tri(Q)^{-1}=\mathbf{b}$,
and that this subgroup is an infinite cyclic group generated
by a special element $\widetilde{z}$ of $\Br_Q$, called \emph{fundamental element}.
The center $Z(\Br_Q)$ of $\Br_Q$ is an infinite cyclic group.
The element $z_Q=\widetilde{z}$ generates $Z(\Br_Q)$ if $Q$ is of type $D_n$ for even $n$,
and $z_Q=\widetilde{z}^2$ generates $Z(\Br_Q)$ if $Q$ is of type $D_n$ for odd $n$ or of  type A.
If we denote by $\Gamma$ the Ginzburg dg algebra associated to the quiver $Q$ with zero potential $W=0$, then  there is a quotient map
\[
    \pi:\Br_Q\to\Br\Gamma
\]
since the spherical twists satisfy the braid relation cf. \cite[(7.4)]{KQ}.

\begin{example}\cite{Qiu}
Let $\surf$ be an $(n+3)$-gon as shown in the left picture of Figure~\ref{fig:triangulations}.
Then the tagged rotation has order $n+3$.
Further, for the shift $[1]$ in $\D(\surf)$, we have
\[
    \pi(\widetilde{z}^2)=[n+3].
\]
where $Q$ is a quiver of type $A_n$.
\end{example}

\begin{example}\cite{Qiu}
Similarly, if $\surf$ is an $n$-gon with a puncture as shown in
the right picture of Figure~\ref{fig:triangulations},
then the tagged rotation has order $n$ if $n$ is even and order $2n$ if $n$ is odd.
Further, for the shift $[1]$ in $\D(\surf)$, we have
\[
    \pi(\widetilde{z})=[n],
\]
where $Q$ is a quiver of type $D_n$.
\end{example}

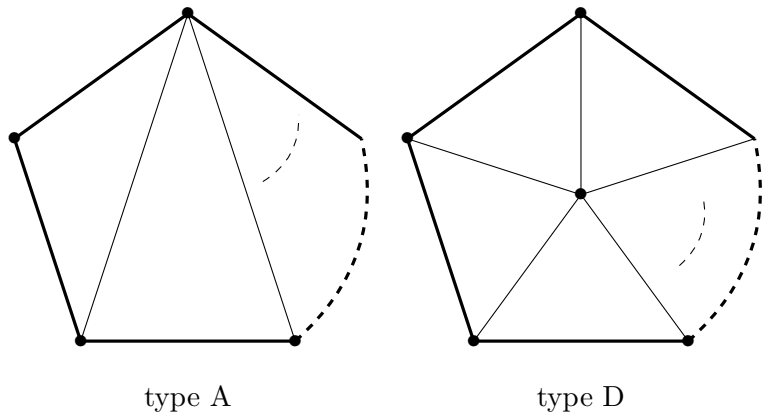
\begin{figure}[b]\centering
\begin{tikzpicture}[scale=.4]
  \draw[very thick, dashed](90-144:6)to[bend right](90-72:6);
  \foreach \j in {3,...,6}{
        \draw[very thick](90-72*\j:6)--(90-72*\j+72:6)node{$\bullet$};}
  \foreach \j in {1,...,4}{  \draw[thin] (90:6)--(90+72*\j:6);}
    \draw(10:2.5)edge[dashed, bend right](36:4.5);
  \draw(-90:6)node[below]{type A};
\end{tikzpicture}
\begin{tikzpicture}[scale=.4]
  \draw[very thick, dashed](90-144:6)to[bend right](90-72:6);
  \foreach \j in {3,...,6}{
        \draw[very thick](90-72*\j:6)--(90-72*\j+72:6)node{$\bullet$};}
  \foreach \j in {1,...,5}{  \draw[thin] (0,0)--(90+72*\j:6);}
    \draw(-36:4)edge[dashed, bend right](0:4);
  \draw(-90:6)node[below]{type D}(0,0)node{$\bullet$};
\end{tikzpicture}
\caption{The triangulations}
\label{fig:triangulations}
\end{figure}

\subsection{Isomorphism Theorem}

In this subsection, we discuss when the  canonical injection
$\iota_\surf\colon\TM(\surf)\hookrightarrow\Aut_0\C{}{\surf}$ from Lemma \ref{lem:MCG} is an isomorphism.
The result can be obtained by cutting the given marked surface along arcs and using induction similar to the methods in section 3.
The details are worked out in \cite{BS}, from where we derive the main result of this subsection:

\begin{theorem}\label{thm:auto}
The  canonical injection
$\iota_\surf\colon\TM(\surf)\hookrightarrow\Aut_0\C{}{\surf}$ is an isomorphism except when $\surf$ is a once-punctured disc with 2 or 4 marked points on the boundary or a twice-punctured disc with 2 marked points on the boundary.
\end{theorem}
\begin{proof}
An element $F$ in $\Aut_0\C{}{\surf}$ acts on the exchange graph $\CEG{*}{\surf}$ of clusters $({\bf x},Q_{\bf x})$, and  it is shown in \cite{ASS} that $F$ maps a quiver $Q_{\bf x}$ of a cluster to an isomorphic quiver. But then it follows from Proposition 8.5 in \cite{BS} that $F$ is induced by an element in the mapping class group $\TM({\surf})$ provided the surface is not one of the exceptions listed in the proposition, hence $\iota_\surf\colon\TM(\surf)\hookrightarrow\Aut_0\C{}{\surf}$ is surjective.
\end{proof}


TB: Bishop's University

2600 College Street
Sherbrooke, Quebec, J1M 1Z7, Canada. (tbruestl@ubishops.ca)

YQ: NTNU

Department of Mathematical Sciences, NTNU,
7491 Trondheim,
Norway. (Yu.Qiu@bath.edu)

\end{document}